\documentclass[11pt]{amsart}

\usepackage{amssymb,amsmath,amsthm}
\usepackage[left=3.2cm,top=3.5cm,right=3.2cm,bottom=3cm]{geometry}
\usepackage{graphicx}
\usepackage{color}

\def\XXint#1#2#3{{\setbox0=\hbox{$#1{#2#3}{\int}$}
\vcenter{\hbox{$#2#3$}}\kern-.5\wd0}}

\graphicspath{./}

\theoremstyle{definition}
\newtheorem{theorem}{Theorem}[section]
\newtheorem{lemma}[theorem]{Lemma}
\newtheorem{proposition}[theorem]{Proposition}

\newtheorem{definition}[theorem]{Definition}

\newtheorem{remark}[theorem]{Remark}

\numberwithin{equation}{section}

\newcommand{\R}{\mathbb{R}}
\newcommand{\C}{\mathbb{C}}
\renewcommand{\H}{\mathbb{H}}
\renewcommand{\O}{\mathbb{O}}
\renewcommand{\S}{\mathbb{S}}
\newcommand{\M}{\mathcal{M}}

\newcommand{\eps}{\varepsilon}
\newcommand{\inj}{\text{inj}}
\newcommand{\ra}{\rightarrow}
\newcommand{\vol}{\text{vol}}
\renewcommand{\div}{\text{div}}
\newcommand{\tr}{\text{tr}}

\newcommand{\dvol}{\text{dvol}}
\renewcommand{\P}{\mathbb{P}}

\title{Discrete and Continuous Green Energy on Compact Manifolds}
\date{}
\author{Carlos Beltr\'an}
\email[Carlos Beltr\'an]{beltranc@unican.es}
\author{Nuria Corral}
\email[Nuria Corral]{corraln@unican.es}
\author{Juan G. Criado del Rey}
\email[Juan G. Criado del Rey]{gonzalezcj@unican.es}
\address[Postal address of the three authors]{Dpto. de Matem\'aticas, Estad\'istica y Computaci\'on, Facultad de Ciencias, Universidad de Cantabria\\
Avda. Los Castros s/n\\
39005 Santander, Cantabria (SPAIN).}
\thanks{This research was partially supported by Ministerio de Econom\'ia y Competitividad under grants MTM2014-57590-P, MTM2013-46337-C2-1-P and a PhD grant by the University of Cantabria.}

\begin{document}
\begin{abstract}
In this article we study the role of the Green function for the Laplacian in a compact Riemannian manifold as tool for obtaining well--distributed points. In particular, we prove that a sequence of minimizers for the Green energy is asymptotically uniformly distributed. We pay special attention to the case of locally harmonic manifolds.
\end{abstract}

\maketitle

\section{Introduction}
Distributing points in spheres or other sets is a very classical problem. Its modern formulation in terms of energy--minimizing configurations is due to the discoverer of the electron J. J. Thomson who in 1904 posed the question (rephrased here) {\em in which position --within some set such as a ball or a sphere--  would $N$ electrons lie in order to minimize their electrostatic potential?} (see \cite{thomson}). The actual origin of the problem is certainly much older (see Section \ref{sec:mit}).

Thomson's question was related to a certain atomic model -- the plumb pudding model --  which had a very short life due to the spectacular advances of experimental physics in the beginning of the XX century. The question still remained as a beautiful problem to be solved, and gained importance for different applications in the subsequent years. In 1930, the botanist Tammes suggested that the (astonishingly regular) distribution of pores in pollen particles followed a pattern that maximized the minimal distance between pores (see \cite{tammes1930} for Tammes' original publication and \cite{pollen} for high definition images). This idea gave an excellent explanation to the fact that there are barely pollen particles with $5$ or $11$ pores (since if $5$ pores can be placed in the surface of a sphere then $6$ equal sized pores can also be placed, and similarly for $11$ and $12$. The mathematical proof of this fact was not complete until the 1980's, see \cite{FejesToth,SvdW,Boroczky,Danzer}). See \cite{whyte,leech1957} for two classical reviews on the problem. 

A seminal paper \cite{SaffKuijlaars} launched a new collection of works on the topic of distributing points in spheres. The problem had gain new motivation with Shub and Smale's approach to polynomial system solving, which in the one--dimensional case required to find a polynomial all of whose zeros were well--conditioned in a particular sense. In \cite{bezout_iii} they proved that such zeros correspond (via the stereographic projection) to points in the Riemann sphere which maximize the product of their mutual distances, equivalently, points with minimal logarithmic energy. This relation led Smale to include the problem of algorithmically finding these points in his list of problems for the XXI Century \cite{Smale_problems}. See \cite{stateofart,Nerattinietal} for recent surveys on Smale's problem.

There are many different approaches to the definition of what a {\em sensibly distributed} collection of spherical points is. Apart from the mentioned minimization of the energy and maximization of minimal distances, other definitions include having small discrepancy, providing exact integral formulas for low degree polynomials (spherical $t$--designs, see \cite{Bondarenkoetal,ujue} for a recent breakthrough), having optimal covering radius, maximizing the sum of the mutual distances, etc. There are dozens of papers on each of these problems. A very recent and very complete survey on the problem is \cite{BrauchartGrabner}.

In a recent paper \cite{beltran_facloc}, the problem of minimizing the logarithmic energy in the $2$--sphere was rewritten as a facility location problem: that of allocating a number of heat sources in such a way that the average temperature is maximized. This approach led to some nontrivial results including upper bounds for the logarithmic energy of well--separated sequences with small discrepancy. As a consequence it was proved that a sequence of minimizers of Riesz's $s$--energy is asymptotically minimizing for the logarithmic energy (the reciprocal of this fact was proved in Leopardi's paper \cite{Leopardi2013}).

The logarithmic energy is defined by
\[
E_{\mathrm{log}}(x_1, ..., x_N) = \sum_{i\neq j}\log\|x_i-x_j\|^{-1}, \quad x_i\in \S^2.
\] This function has a very special property: its (spherical) Laplacian is constant. This follows from the fact that the function $\log\|x-y\|^{-1}$ is (up to scaling) the Green function for the Laplacian in $\S^2$. A collection of points minimizing the logarithmic energy is called a set of {\em elliptic Fekete points}, though sometimes the word ``elliptic'' is omitted. See \cite{Fekete,Tsuji,HardinSaff} for an introduction to the classical theory.

The Riesz $s$--energy is defined by
\[
E_s(x_1, ..., x_N) = \sum_{i\neq j}\|x_i-x_j\|^{-s}.
\] Remarkable progress in the study of logarithmic and Riesz energies (minimum values, properties of the minimal energy configurations, relation to separation distance, spherical cap discrepancy and cubature formulas...) has taken place in the last three decades, see for example (additionally to the already mentioned references) \cite{CohnKumar} for universally optimal configurations, \cite{saff94,saff98,SandierSerfaty,BeterminSandier,BHS2012,BMO} for asymptotic bounds on the energy, \cite{BorodachovHardinSaff} for complexity considerations on the computation of the energy, \cite{ABD,WomersleySloan,Brauchar2011,Hecke1,Hecke2,ABD,beck1,beck2,GT} for relation to discrepancy, interpolation and quadrature, \cite{dragnev,Leopardi2013} for relation to separation distances.

We now want to define \emph{minimal energy points} in an arbitrary compact manifold. Interesting cases include orthogonal groups (as in \cite{smith}), but we are looking for a standard approach for the general case. General manifolds lack a standard embedding into some Euclidean space, so we cannot directly use the energies defined above in those cases. Still, we could take the definition of the logarithmic energy or the Riesz s--energy and just change Euclidean distance by Riemannian distance. This is something feasible, but if we do so, the resulting function will not be everywhere smooth, since the Riemannian distance function is not everywhere smooth in compact manifolds. A more natural approach would be to take an intrinsic smooth kernel $K(x,y)$ and use it to define an intrinsic discrete energy with the formula $\sum_{i\neq j}K(x_i,x_j)$. In this article we will study the role of the Green function $G$ of a compact manifold as a measure of the well--distribution of a set of points. The \emph{(discrete) Green energy} of a set of $N > 1$ distinct points $x_1, ..., x_N \in \M$ is given by
\[
E_G(x_1, ..., x_N) = \sum_{i\neq j}G(x_i,x_j).
\] This definition leads in fact to well--distributed points in the following sense:

\begin{theorem}[Main]\label{th:main} Let $\M$ be a compact Riemannian manifold of dimension $n > 1$ and let $G$ be its Green function. The unique probability measure minimizing the continuous $G$--energy
\[
\mathcal{I}_G[\mu] = \int_{x,y\in\M}G(x,y)d\mu(x)d\mu(x)
\] is the uniform measure $\lambda$ on $\M$. Moreover, for each $N > 1$, let $\omega^*_N = \{x_1, ..., x_N\}$ be a set of minimizers for the Green energy $E_G$. Then
\[
\frac{1}{N}\sum_{x\in \omega^*_N} \delta_{x} \stackrel{*}{\rightharpoonup} \lambda.
\]
\end{theorem}

The Green function of a general Riemannian manifold is usually very hard or almost impossible to compute. However, there is a class of manifolds in which the Green function can be computed by solving a simple ODE: the class of locally harmonic manifolds. In Section \ref{section:lh} we work out some examples of Green functions for this class of manifolds.

\subsection{The problem in classical mythology}\label{sec:mit}

 In the Metamorphoses, Ovid tells the story of the nymph Io, seen by Jupiter when she was walking through the forest. Jupiter, fascinated by her beauty, assaulted her and covered the trees with a dense cloud so that his wife Juno wouldn't catch them. The goddess suspected something and run towards her husband, but he quickly converted Io into a beautiful, white cow.

Still suspicious, Juno asked Jupiter to give her the cow as a gift, which he accepted to avoid further questioning. Juno then ordered the giant Argos Panoptes to watch the cow day and night. Argos has got a hundred eyes evenly distributed along his body so that some of the eyes might be closed while still watching on every direction.

Jupiter finally sent Mercury to rescue Io. By telling the giant a long story, presumably boring but actually quite charming, the god made Argos fall asleep in such a way that he closed all his eyes at once. He then slew the giant and released Io. After that she began a long journey around the world until she was forgiven by Juno and returned to her original shape.

\begin{figure}[htp]
	\centering
		\includegraphics[width=0.9\textwidth]{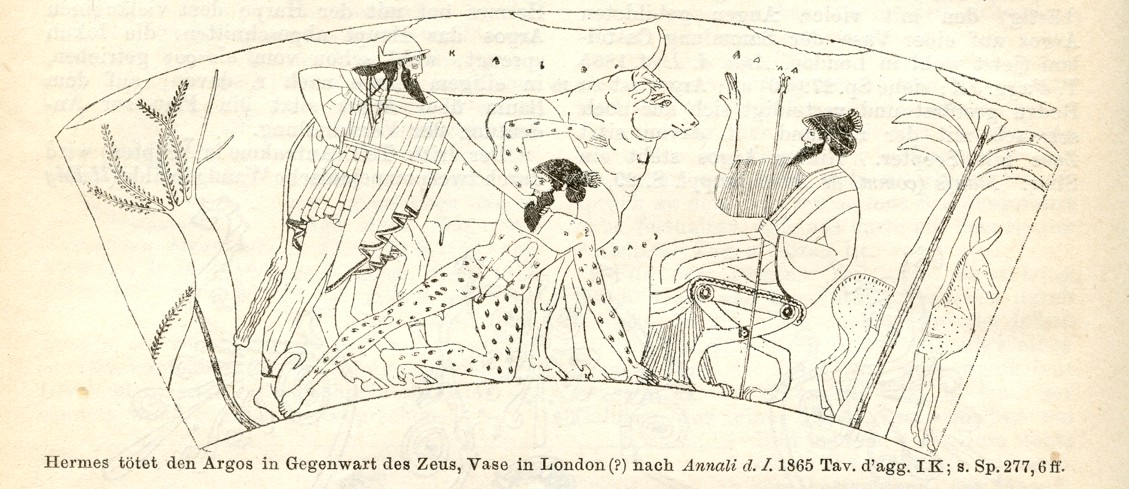}
	\caption{A picture by Wilhelm Heinrich Roscher, dated 1890, of a 5th century BCE Athenian red figure vase depicting Hermes slaying Argos Panoptes. Note the eyes which are evenly distributed on Argo's body. Io as a cow observes the scene in the background. The picture, of public domain, was found in \cite{wiki}.} \label{fig:argos}
\end{figure}

The Greek artists who painted the scene in the 5th century BCE allocated a large number of eyes on the body of Argos; the criteria had probably been a reasonable separation distance and a short covering radius. A vase with such painting can be found on the Kunsthistorisches Museum in Vienna, see Figure \ref{fig:argos}.   

\section{Green Function}

\subsection{Notation and conventions.} Through this article $\M = (\M,g)$ will denote a $C^\infty$ compact Riemannian manifold without boundary of dimension $n > 1$ with volume $V = \vol(\M)$ and volume form $\dvol$. The (Riemannian) distance between any pair of points $x,y\in \M$ will be denoted by $d(x,y)$. For a point $x\in \M$ we will denote by $\inj(x)$ the injectivity radius of $x$. The injectivity radius of $\M$ will be denoted by $\inj(\M)$. We will denote by $B(x,r) = \{y\in \M : d(x,y) < r\}$ the geodesic ball of radius $r$ and by $S(x,r) = \{y\in \M : d(x,y) = r\}$ the geodesic sphere of radius $r$. We will follow the convention that the Riemannian Laplacian is given by $\Delta = -\div \nabla$.

\subsection{The Green function in a compact manifold}
We will start by recalling the existence of the Green function in a compact manifold.

\begin{theorem}\cite[Theorem 4.13]{aubin}\label{thm:green_f} Let $\M^n$ be a compact Riemannian manifold. There exists a smooth function $G$ defined on $\M \times \M$ minus the diagonal with the following properties:
\begin{enumerate}
\item For every function $\varphi\in C^2(\M)$,
\[
\varphi(x) = V^{-1}\int_{y\in \M} \varphi(y)\dvol(y) + \int_{y\in \M}G(x,y)\Delta\varphi(y)\dvol(y).
\] In other words,
\[
\Delta_yG(x,y) = \delta_x(y)-V^{-1}
\] in the sense of distributions.

\item There exists a constant $k$ such that, for every $x\neq y$,
\[
\begin{array}{lr}
|G(x,y)| \leq k(1+|\log r|) & \text{ for } n = 2,\vspace{0.2cm}\\
|G(x,y)| \leq kr^{2-n} & \text{ for } n > 2,\vspace{0.2cm}\\
\|\nabla_yG(x,y)\| \leq kr^{1-n},\vspace{0.2cm}\\
\|\nabla^2_yG(x,y)\| \leq kr^{-n},
\end{array}
\] with $r = d(x,y)$.

\item There exists a constant $A$ such that $G(x,y) \geq A$.

\item The function $x\mapsto \int_{y\in\M}G(x,y)\dvol(y)$ is constant.

\item $G(x,y) = G(y,x)$ for every $x\neq y$.
\end{enumerate}
\end{theorem}

We call $G$ the \emph{Green function for the Laplacian}. The Green function is uniquely defined by (1) up to an additive constant. Hence, in view of (4) in the previous theorem, we may assume that $\int G = 0$.

\begin{remark} One can think of the Green function as follows: suppose that we place a source of infinite heat at a point $x\in \M$ and that $\M$ is cooling at a constant rate $V^{-1}$. Then $y\mapsto G(x,y)$ is a stationary solution to the heat equation $(\partial_t+\Delta)u = \delta_x-V^{-1}$.
\end{remark}

\begin{remark}\label{rmk:poisson} From (1) in Theorem \ref{thm:green_f}, if $f:\M \ra \R$ is a continuous function with $\int f = 0$, then the function
\[
u(x) = \int_{y\in \M} G(x,y)f(y)\dvol(y)
\] satisfies $\Delta u = f$.
\end{remark}

\subsection{Locally harmonic manifolds}\label{section:lh} As we said before, computing the Green function of a general Riemannian manifold is usually a hard task. However, if we restrict ourselves to the class of locally harmonic manifolds, the computations are much easier. A manifold is locally harmonic at a point $x$ if every sufficiently small geodesic sphere around $x$ has constant mean curvature.

A more useful (yet equivalent) definition for our purposes uses the concept of volume density.

\begin{definition}\label{def:density} Let $\M$ be a Riemannian manifold and let $x\in \M$ be a point. The \emph{volume density} $\omega_x$ is a function whose local expression in normal coordinates around $x$ is
\[
\omega_x(y) = \sqrt{\det{g_{ij}(\exp_x^{-1}(y))}}.
\] (See \cite[6.3]{besse} or \cite{kreyssig} for a coordinate--free definition).
\end{definition}

\begin{proposition}\label{prop:density} The volume density satisfies the following properties:
\begin{enumerate}
\item $\omega_x$ is smooth in any normal neighbourhood around $x$.
\item $\omega_x(x) = 1$.
\item $\omega_x(y) > 0$ if $d(x,y) < \inj(x)$.
\item $\omega_x(y) = 0$ if and only if $y$ is a conjugate point to $x$.
\item $\omega_x(y) = \omega_y(x)$.
\end{enumerate}
\end{proposition}

Properties (1), (2) and (3) follow directly from the local expression of $\omega_x$ (for (2), recall that in normal coordinates around $x$, $g_{ij}(0) = \delta_{ij}$). Property (4) follows from the definition of the volume density in terms of Jacobi fields (see \cite[6.3]{besse}): $\omega_x(y) = 0$ when there is a non--zero Jacobi field vanishing at $x$ and $y$. Finally, (5) is a direct consequence of the invariance of $\omega$ under the canonical geodesic involution (see Definition \ref{def:involution} and Lemma \ref{lem:involution} below).

\begin{definition}\label{def:involution} The \emph{canonical geodesic involution} $i$ on the tangent bundle $T\M$ of $\M$ is defined as follows. Denote by $\Omega\subset T\M$ the domain of definition of the exponential map and by $\gamma_v$ the geodesic $\gamma_v(t) = \exp_x(tv)$, $t\in [0,1]$, where $v\in T_x\M$. Then $i:\Omega \ra \Omega$ is the map
\[
i(v) = -\dot{\gamma}_v(1).
\]
\end{definition}

In other words, if $v\in T_x\M$ and $y = \exp_xv$, then $i(v)$ is the unique vector $w\in T_y\M$ such that $x = \exp_yw$.

\begin{lemma}\cite[Lemma 6.12]{besse}\label{lem:involution} Let $v\in T_x\M$ and $y = \exp_xv$. Let us denote, abusing of notation, $\omega(v) = \omega_x(y)$. Then $\omega(v) = \omega(i(v))$.
\end{lemma}

\begin{definition} Let $\M$ be a Riemannian manifold and let $x\in \M$ be a point. We will say that $\M$ is \emph{locally harmonic at $x$} if there exists an $\eps > 0$ such that $\omega_x$ is radially symmetric on $B(x,\eps)$. In other words, if there is a function $\Omega_x:[0,\eps) \ra \R$ such that $\omega_x(y) = \Omega_x(d(x,y))$ for every $y\in B(x,\eps)$. We will say that $\M$ is \emph{locally harmonic} if it is locally harmonic at $x$ for every $x\in \M$.
\end{definition}

The Euclidean space $\R^n$ is a simple example of a locally harmonic manifold with $\Omega_x(r) \equiv 1$. The sphere $\S^n$ is also locally harmonic with $\Omega_x(r) = \sin^{n-1}r/r^{n-1}$. Other examples of locally harmonic manifolds are the projective spaces $\R\P^n$, $\C\P^n$, $\H\P^n$ and $\O\P^2$.

\begin{remark} Every locally harmonic manifold of dimension $n > 2$ is an Einstein manifold (see \cite[Chapter 6]{besse}). As a consequence of a theorem by DeTurck--Kazdan \cite[Theorem 5.2]{deturck}, the representation of the metric $g$ in normal coordinates is real analytic. Thus the volume density in normal coordinates is also real analytic. Therefore, if $\omega_x$ is radially symmetric on $B(x,\eps)$ for some $\eps > 0$, then it is also radially symmetric on $B(x,\inj(x))$. In other words, we can take $\eps = \inj(x)$. Moreover, if $\M$ is locally harmonic and $d(x,y) < \inj(\M)$, then by \cite[Proposition 6.16]{besse}, $\omega_x$ and $\omega_y$ are radially symmetric with the same function $\Omega$. This means that we can drop the subscript $x$ in $\Omega_x$.
\end{remark}

Recall (\cite[4.9]{aubin}) that in Riemannian polar coordinates the Laplacian of a radially symmetric function $f(r)$ is given by
\[
-\Delta f(r) = f''(r)+\frac{n-1}{r}f'(r)+f'(r)\partial_r\log\sqrt{\det{g_{ij}(r\theta)}},
\] thus if $\M$ is locally harmonic, then the Laplacian of the distance function $d_x = d(x,\cdot)$ in normal coordinates around $x$ is
\begin{equation}\label{eq:laplacian_radial}
L_x(r):= \Delta d_x(y) = -\frac{n-1}{r}-\frac{\Omega'(r)}{\Omega(r)} = -\frac{d}{dr}\log r^{n-1}\Omega(r), \quad r = d(x,y).
\end{equation}

\begin{remark} Since the right hand side of \eqref{eq:laplacian_radial} does not depend on $x$, we can drop the subscript $x$ in $L_x(r)$ and simply write $L(r)$ as long as $r < \inj(x)$.
\end{remark}

The connection between locally harmonic manifolds and the mean curvature of the geodesic spheres being constant is now clarified in the next proposition. If $x,y\in\M$ are two points with $r = d(x,y) < \inj(x)$, let us denote by $\sigma_{x,y}$ the shape operator of the geodesic sphere $S(x,r)$ at the point $y$. That is, $\sigma_{x,y}(X) = -\nabla_XN(y)$ for any vector field $X$ tangent to $S(x,r)$, where $N = \partial_r$ is the unit outward normal vector field along $S(x,r)$.

\begin{proposition} Let $x,y\in \M$ be two distinct points with $r = d(x,y) < \inj(x)$. Then
\[
\Delta d_x(y) = \tr\,\sigma_{x,y}(y).
\]
\end{proposition}

\begin{proof} Let $\{E_1, ..., E_{n-1}, N\}$ be an orthonormal frame around $y$. Then
\[
\Delta d_x = -\div\nabla d_x = -\tr\nabla N = -\sum_{i=1}^{n-1}g(\nabla_{E_i}N,E_i)-g(\nabla_NN,N) = \tr\,\sigma_{x,y}-g(\nabla_NN,N).
\] But $g(\nabla_NN,N) = \frac{1}{2}Ng(N,N) = 0$. The result follows.
\end{proof}

Recall that the mean curvature of the geodesic sphere at $y$ is defined as $\frac{1}{n-1}\tr\,\sigma_{x,y}$. Therefore, if $\M$ is locally harmonic, then $\Delta d_x$ is radially symmetric, so $\tr\,\sigma_{x,y}$ depends only on $r = d(x,y)$ and hence the mean curvature is constant along $S(x,r)$.

We will prove one more result about $\Delta d_x$ in a locally harmonic manifold.

\begin{proposition}\label{prop:der_vr} Let $\M$ be locally harmonic at $x$ and $0< r < \inj(x)$. Let us denote by $v_x(r)$ the volume of the geodesic sphere $S(x,r)$. Then
\[
-L(r) = \frac{d}{dr}\log v_x(r).
\]
\end{proposition}

\begin{proof} Let $0 < \delta < r$ be any real number, and denote
\[
A(x,\delta,r) = \{y\in \M : \delta \leq d(x,y) \leq r\}.
\] From the Divergence Theorem,
\begin{align*}
-\int_{y\in A(x,\delta,r)}\Delta d_x(y)\dvol(y) &= \int_{y\in S(x,r)}\|\nabla d_x(y)\|^2\text{d}S_r(y)-\int_{y\in S(x,\delta)}\|\nabla d_x(y)\|^2\text{d}S_\delta(y)\\
&= v_x(r)-v_x(\delta),
\end{align*} since $\|\nabla d_x(y)\|\equiv 1$ because $d_x$ is a distance function. Making $\delta \ra 0$, we get on one hand
\[
\int_{y\in B(x,r)}\Delta d_x(y)\dvol(y) = -v_x(r).
\] On the other hand, taking Riemannian polar coordinates around $x$,
\begin{align*}
\int_{y\in B(x,r)}\Delta d_x(y)\dvol(y) &= \int_0^r \int_{y\in S(x,t)}\Delta d_x(y)dS_t(y)dt\\
&= \int_0^r L(t)v_x(t)dt.
\end{align*} Differentiating, we get
\[
-\frac{d}{dr}v_x(r) = L(r)v_x(r),
\] or
\[
-L(r) = \frac{d}{dr}\log v_x(r).
\]
\end{proof}

\begin{remark} Since the left hand side in the equality of Proposition \ref{prop:der_vr} does not depend on the point $x$, we can drop the subscript $x$ in $v_x(r)$ and simply write $v(r)$ as long as $r < \inj(x)$.
\end{remark}

Further equivalent conditions to local harmonicity can be found in \cite[Proposition 6.21]{besse}.

\subsection{The Green function in a locally harmonic manifold} A simple computation yields the formula for the Laplacian of the composition of two functions $\M \stackrel{f}{\rightarrow} \R \stackrel{\phi}{\rightarrow} \R$:
\[
\Delta(\phi\circ f) = -(\phi''\circ f)\|\nabla f\|^2+(\phi'\circ f)\Delta f.
\] If $f = d_x$ and $\M$ is locally harmonic, then according to Theorem \ref{thm:green_f} we can compute the Green function of $\M$ near $x$ by solving the ODE
\begin{equation}\label{eq:the_ode}
\phi''-\phi'L =V^{-1}, \qquad 0 < r < \inj(x).
\end{equation} The integrating factor for this equation is
\[
q(r) = e^{\int -L(r)dr} \stackrel{\text{Prop. \ref{prop:der_vr}}}{=} e^{\log v(r)} = v(r).
\] Hence, the general solution is
\begin{equation}\label{eq:general_sol}
\phi'(r) = \frac{V^{-1}\int v(r)dr+C}{v(r)}, \qquad 0 < r < \inj(x).
\end{equation}

Now we would like to recover $G$ from $\phi$ by setting $G(x,y) = \phi(d(x,y))$. If we write $G$ in this form, then $y\mapsto G(x,y)$ is defined as far as $d(x,y) < \inj(x)$, which is the distance from $x$ to the closest point in its cut locus $\text{Cut}(x)$. If there is some point $z$ further than $\inj(x)$ from $x$, then $G(x,z)$ is not defined. If every point in $\text{Cut}(x)$ was at the same distance $\inj(x)$ from $x$, then we would be able to extend $\phi(d(x,y))$ to the remaining points of $\M$ by continuity. The perfect candidates for recovering $G$ in this fashion are Blaschke manifolds.

\begin{definition}\label{def:blaschke} Let $\M$ be a compact Riemannian manifold. We will say that $\M$ is a \emph{Blaschke manifold} if $\inj(\M) = \text{diam}(\M)$.
\end{definition}

\begin{proposition}\label{prop:blaschke} If $\M$ is a Blaschke manifold, then for every $x\in \M$ and for every $y\in \text{Cut}(x)$, we have that $d(x,y) = \text{diam}(\M)$.
\end{proposition}

\begin{proof} Let $x\in \M$ be any point and let $y\in \text{Cut}(x)$. Then
\[
\inj(\M) \leq \inj(x) \leq d(x,y) \leq \text{diam}(\M),
\] and all these numbers are equal.
\end{proof}

The following lemma provides sufficient conditions to extend the function $\phi$ to a $C^2$ function defined on $\M$. The proof is standard and it can be found, for example, in \cite[Lemma 4.2.2]{kreyssig}.

\begin{lemma}\label{lem:c2_blaschke} Let $\M$ be a Blaschke manifold of diameter $D = \text{diam}(\M)$ and let $f:(0,D] \ra \R$ be a continuous function which is $C^2$ on $(0,D)$. If $\lim_{r\ra D^-}f'(r) = 0$, then the function $F(x,y) = f(d(x,y))$ is $C^2$ on $\M \times \M$ minus the diagonal.
\end{lemma}

\begin{theorem}\label{th:green_lh} Let $\M$ be a locally harmonic, Blaschke manifold. Then the Green function of $\M$ is given by $G(x,y) = \phi(d(x,y))$, where
\[
\phi'(r) = -\frac{V^{-1}\int_r^{\inj(\M)}v(t)dt}{v(r)}
\] and $\phi$ is a primitive of $\phi'$ making $\int \phi(d(x,y)) = 0$.
\end{theorem}

In order to prove Theorem \ref{th:green_lh}, we will need some previous results. Let us denote by $D = \text{diam}(\M) = \inj(\M)$ for $\M$ a Blaschke manifold.

\begin{lemma}\label{lem:vr_omega} Let $\M$ be locally harmonic and Blaschke. Then
\[
v(r) = \vol(\S^{n-1})r^{n-1}\Omega(r), \quad 0 < r < D.
\]
\end{lemma}

\begin{proof} Computing in normal coordinates,
\[
v(r) = \int_{\theta\in \mathbb{S}(r)}\Omega(r)d\theta = \vol(\S^{n-1})r^{n-1}\Omega(r).
\]
\end{proof}

\begin{lemma}\label{lem:limit_phi} Let $\phi$ be as in Theorem \ref{th:green_lh}. Then
\begin{enumerate}
\item $\lim_{r\ra 0}r^{n-1}\phi'(r) = -\frac{1}{\vol(\S^{n-1})}$.
\item $\lim_{r\ra 0}r^{n-1}\phi(r) = 0.$
\end{enumerate}
\end{lemma}

\begin{proof} From Lemma \ref{lem:vr_omega},
\[
\lim_{r\ra 0}r^{n-1}\phi'(r) = \lim_{r\ra 0} \frac{v(r)}{\vol(\S^{n-1})\Omega(r)}\phi'(r) = -\lim_{r\ra 0}\frac{V^{-1}\int_r^D v(t)dt}{\vol(\S^{n-1})\Omega(r)} = -\frac{1}{\vol(\S^{n-1})},
\] since $\Omega(0) = 1$ and $\int_0^D v(t)dt = V$. This proves (1).

Now, by L'H\^opital's rule,
\[
\lim_{r\ra 0}r^{n-1}\phi(r) = \lim_{r\ra 0}\frac{\phi'(r)}{-\frac{n-1}{r^n}} = -\frac{1}{n-1}\lim_{r\ra 0}r\cdot r^{n-1}\phi'(r) = 0,
\] by (1).
\end{proof}

\begin{lemma}\label{lem:phi_integrable} Let $\phi$ be as in Theorem \ref{th:green_lh}. Then the function $y\mapsto \phi(d(x,y))$ is integrable.
\end{lemma}

\begin{proof} Since $\M$ is Blaschke, computing in normal coordinates we have that
\begin{align*}
\int_{y\in \M}|\phi(d(x,y))|\dvol(y) &= \int_{y\in B(x,D)}|\phi(d(x,y))|\dvol(y)\\
&= \int_0^D \int_{\theta\in\mathbb{S}^{n-1}(r)}|\phi(r)|\Omega(r)d\theta dr\\
&= \int_0^D \vol(\S^{n-1})r^{n-1}|\phi(r)|\Omega(r)dr,
\end{align*} which is finite because $\lim_{r\ra 0}r^{n-1}\phi(r) = 0$ from Lemma \ref{lem:limit_phi}.
\end{proof}

\begin{proof}[Proof of Theorem \ref{th:green_lh}] We first check that $\phi'$ can be extended continuously to $r = D$. That is, that the limit $\lim_{r\ra D^-}\phi'(r)$ exists. If $\lim_{r\ra D^-}v(r) \neq 0$, then clearly $\lim_{r\ra D^-}\phi'(r) = 0$. Assume that $\lim_{r\ra D^-}v(r) = 0$. We can apply L'H\^opital's rule to compute
\[
\lim_{r\ra D^-}\phi'(r) =  \lim_{r\ra D^-}\frac{V^{-1}v(r)}{v'(r)} = \lim_{r\ra D^-}V^{-1}\frac{1}{\frac{d}{dr}\log v(r)} = 0,
\] since $v(r) \ra 0$, $\log v(r) \ra -\infty$ and then $\frac{d}{dr}\log v(r) \ra +\infty$. Thus, not only $\lim_{r\ra D^-}\phi'(r)$ exists, but also it equals $0$. By Lemma \ref{lem:c2_blaschke}, the function $F(x,y) = \phi(d(x,y))$ is $C^2$ on $\M \times \M$ minus the diagonal. Being a solution of \eqref{eq:the_ode}, $\Delta_yF(x,y) = -V^{-1}$ if $d(x,y) < D$, but because it is $C^2$ on $\M$, $\Delta_yF(x,y) = -V^{-1}$ on $\M \times \M$ minus the diagonal.

Now let $u \in C^2(\M)$ be any function. Let $0 < \eps < \inj(\M)$ be any real number and let us denote $B_\eps = B(x,\eps)$, $B_\eps^c = \M \setminus B(x,\eps)$ and $S_\eps = S(x,\eps)$. Then
\begin{align*}
\int_{y\in \M}\Delta u(y)\phi(d(x,y))\dvol(y) = \int_{y\in B_\eps}\Delta u(y)\phi(d(x,y))\dvol(y)+\int_{y\in B_\eps^c}\Delta u(y)\phi(d(x,y))\dvol(y).
\end{align*} Let $I_1(\eps)$ be the first integral on the right hand side of the equality above and let $I_2(\eps)$ be the second one. From Green's Second Identity, since the outward unit normal to $B_\eps^c$ is $N = -\partial_r = -\nabla d_x$,
\begin{align*}
I_2(\eps) &= \int_{y\in S_\eps}\left[-u(y)\partial_r\phi(d(x,y))+\phi(d(x,y))\partial_ru(y)\right]dS_\eps(y)+\int_{y\in B_\eps^c}u(y)\Delta_y \phi(d(x,y))\dvol(y)
\end{align*} Again, let $I_3(\eps)$ be the first integral on the right hand side and let $I_4(\eps)$ be the second one. Now, computing in normal coordinates,
\begin{align*}
I_3(\eps) &= -\phi'(\eps)\int_{y\in S_\eps}u(y)dS_\eps(y)+\phi(\eps)\int_{y\in S_\eps}\partial_ru(y)dS_\eps(y)\\
&= -\phi'(\eps)\int_{\theta\in \mathbb{S}^{n-1}(\eps)}u(\exp_x\theta)\Omega(\eps)d\theta+\phi(\eps)\int_{\theta\in\mathbb{S}^{n-1}(\eps)}\partial_ru(\exp_x\theta)\Omega(\eps)d\theta\\
&= -\phi'(\eps)\Omega(\eps)\int_{\theta\in \S^{n-1}}\eps^{n-1}u(\exp_x\eps\theta)d\theta+\phi(\eps)\Omega(\eps)\int_{\theta\in \S^{n-1}}\eps^{n-1}\partial_ru(\exp_x\eps\theta)d\theta\\
&= -\phi'(\eps)\eps^{n-1}\Omega(\eps)\int_{\theta\in \S^{n-1}}u(\exp_x\eps\theta)d\theta+\phi(\eps)\eps^{n-1}\Omega(\eps)\int_{\theta\in \S^{n-1}}\partial_ru(\exp_x\eps\theta)d\theta\\
\end{align*} Since $\Omega(0) = 1$ and, from Lemma \ref{lem:limit_phi} we have
\[
\lim_{\eps\ra 0}-\phi'(\eps)\eps^{n-1} = \frac{1}{\vol(\S^{n-1})}, \quad \text{ and } \quad \lim_{\eps\ra 0}\phi(\eps)\eps^{n-1} = 0,
\] we conclude that
\[
\lim_{\eps \ra 0}I_3(\eps) = u(x).
\] Since $y\mapsto \phi(d(x,y))$ is integrable by Lemma \ref{lem:phi_integrable} and $u$ bounded,
\[
\lim_{\eps \ra 0}I_1(\eps) = 0.
\] Finally, if $y\in B_\eps^c$, then $\Delta_y\phi(d(x,y)) = -V^{-1}$, so
\[
\lim_{\eps \ra 0}I_4(\eps) = -V^{-1}\int_{y\in \M}u(y)\dvol(y).
\] For every $0 < \eps < \inj(\M)$ we have that
\[
\int_{y\in \M}\Delta u(y)\phi(d(x,y))\dvol(y) = I_1(\eps)+I_3(\eps)+I_4(\eps).
\] In particular, the integral on the left hand side equals
\[
\lim_{\eps \ra 0}(I_1(\eps)+I_3(\eps)+I_4(\eps)) = u(x)-V^{-1}\int_{y\in \M}u(y)\dvol(y).
\] But according to Theorem \ref{thm:green_f}, this is exactly how the Green function acts on $C^2$ functions. Since the Green function is uniquely defined by this action up to a constant and we assume that both $\phi(d(x,y))$ and $G(x,y)$ have zero integral, necessarily $\phi(d(x,y)) = G(x,y)$.
\end{proof}

\begin{remark} In the proof of Theorem \ref{th:green_lh} we distinguished between the cases
\[
\lim_{r\ra D^-}v(r) = 0 \quad \text{and} \quad \lim_{r\ra D^-}v(r) \neq 0.
\] An example of the first case is the sphere $\S^2$, where the geodesic spheres $S(x,r)$ collapse to a point when $r \ra \pi^-$. The situation is different in the case of $\R\mathbb{P}^2$, for example. If we think of the half--sphere model of $\R\P^2$, geodesic spheres departing from the north pole grow in volume until they reach the equator, which is the cut locus of the north pole, and then they go back again until they collapse to the north pole. The reason for this is that the points in the cut locus in the case of $\S^2$ are conjugate points (hence $\Omega(D) = 0$), while these points are not conjugate in $\R\P^2$ (thus $\Omega(D) \neq 0$).
\end{remark}

\subsection{Some examples} The Compact Rank One Symmetric Spaces (CROSS) are known to be examples of locally harmonic Blaschke manifolds (see \cite[6.18]{besse}). These spaces were classified by \'E. Cartan and they are the sphere $\S^n$, the projective spaces $\R\P^n$, $\C\P^n$ and $\H\P^n$, and the Cayley plane $\O\P^2$. No other examples of locally harmonic Blaschke manifolds are known. In fact, the Lichnerowicz Conjecture claims that the CROSS are the only Riemannian manifolds of this kind.

All we have to know in order to compute the Green function in the CROSS is the corresponding volume density and then use Theorem \ref{th:green_lh}, since, from Lemma \ref{lem:vr_omega}, $v(r) = \vol(\S^{n-1})r^{n-1}\Omega(r)$ and thus
\[
\phi'(r) = -\frac{V^{-1}\int_r^Dt^{n-1}\Omega(t)dt}{r^{n-1}\Omega(r)}.
\] In \cite[Proposition 3.3.1]{kreyssig} these densities have been computed and the results are shown in Figure \ref{fig:table_w} (we assume that the diameter of the  projective spaces is equal to $\pi/2$).

\begin{figure}[h]
\centering
\bgroup
\def\arraystretch{1.5}
\begin{tabular}{|l|r|}
\hline
$\M$ & $r^{n-1}\Omega(r)$\\
\hline
\hline
$\S^n$ & $\sin^{n-1}r$\\
\hline
$\R\P^n$ & $2^{n-1}\sin^{n-1}r$\\
\hline
$\C\P^n$ & $2^{2n-1}\sin^{2n-1}r\cos r$\\
\hline
$\H\P^n$ & $2^{4n-1}\sin^{4n-1}r\cos^3r$\\
\hline
$\O\P^2$ & $2^{15}\sin^{15}r\cos^7r$\\  
\hline
\end{tabular}
\egroup
\caption{Volume densities of the CROSS.}\label{fig:table_w}
\end{figure}

In particular, for $\M = \S^2$, we have
\[
\phi'(r) = -\frac{(4\pi)^{-1}\int_r^\pi \sin t\,dt}{\sin r} = -\frac{1+\cos(r)}{4\pi\sin r},
\] and
\[
\phi(r) = -\frac{1}{2\pi}\log\sin \frac{r}{2}+C
\] In terms of Euclidean distance,
\[
G(x,y) = \frac{1}{2\pi}\log\|x-y\|^{-1}+C,
\] which is the same (up to rescaling) as the logarithmic kernel.

More generally, using the Gauss hypergeometric function and some classical transformation formulas (see \cite[8.391 and 9.131]{losrusos}) one can get the expression for the Green function of $\S^n$ in terms of the Euclidean distance: $G(x,y) = \hat{\phi}(\|x-y\|)$, where
\[
\hat{\phi}(t) = \frac{1}{\vol(\S^n)}\int_{\frac{t^2}{4}}^1 \frac{\text{B}_{1-x}\left(\frac{n}{2},\frac{n}{2}\right)}{(x-x^2)^{n/2}}dx + C = \frac{2}{n\vol(\S^n)}\int_{\frac{t^2}{4}}^1 {}_2F_1\left(1,n,\frac{n}{2}+1,1-s\right)ds + C.
\]

We can also compute the Green function for the projective spaces (in terms of the intrinsic distance). For instance, in the case of $\C\P^3$ and $\C\P^4$, we have that $G(x,y) = \phi(d(x,y))$, where
\[
\phi(r) = \frac{1}{24\vol(\C\P^3)}\left(\frac{1}{\sin^4r}+\frac{2}{\sin^2r}-4\log\sin r\right)+C
\] for $\C\P^3$ and
\[
\phi(r) = \frac{1}{96\vol(\C\P^4)}\left(\frac{2}{\sin^6r}+\frac{3}{\sin^4r}+\frac{6}{\sin^2r}-12\log\sin r\right)+C
\] for $\C\P^4$.

\section{Proof of Theorem \ref{th:main}}

Our aim is to prove that if $\omega_N^*$ is a sequence of minimizers for the Green energy, then the normalized counting measure $\frac{1}{N}\sum_{x\in \omega_N^*}\delta_x$ converges weak$^*$ to the Riemannian (uniform) probability measure $d\lambda = V^{-1}\dvol$. That is to say that the $\omega_N^*$ are asymptotically uniformly distributed. To this end, we will make use of classical potential theory, taking \cite{book} as a primary reference. Although in \cite{book} all the results are stated for an infinite compact subset of $\R^n$, all the proofs are valid for a Riemannian manifold of positive dimension.

\subsection{Energies and equilibrium measures} Let $\M$ be a compact manifold. A \emph{kernel} is a map $K:\M\times \M \ra \R\cup\{+\infty\}$. If $K$ is symmetric and lower semicontinuous, we define the \emph{(discrete) $K$--energy} of a multiset $\omega_N = \{x_1, ..., x_N\}$ by
\[
E_K(x_1, ..., x_N) = \sum_{i\neq j}K(x_i,x_j).
\] We denote the minimum discrete energy by
\[
\mathcal{E}_K(N) = \min_{\omega_N}E_K(\omega_N).
\] The limit
\[
\tau_K(\M) = \lim_{N\ra \infty}\frac{\mathcal{E}_K(N)}{N^2},
\] (which might be infinite) always exists (see \cite{book}) and is called the \emph{$K$--transfinite diameter} of $\M$. If $\mu$ is some Borel probability measure supported on $\M$, then the \emph{$K$--potential} of $\mu$ is the function defined by
\[
U_K^\mu(x) = \int_{y\in \M}K(x,y)d\mu(y).
\] The \emph{(continuous) $K$--energy} of $\mu$ is the value
\[
\mathcal{I}_K[\mu] = (\mu\otimes\mu)(K) = \int_{x,y\in \M}K(x,y)d\mu(x)d\mu(y) = \int_{x\in \M}U_K^\mu(x)d\mu(x).
\] An \emph{equilibrium measure} (if it exists) is a Borel probability measure $\mu^*$ minimizing the continuous $K$--energy. That is,
\[
\mathcal{I}_K[\mu^*] = \inf_{\mu}\mathcal{I}_K[\mu] =: W_K(\M),
\] where the infimum is taken among all Borel probability measures supported on $\M$. We call $W_K(\M)$ the \emph{Wiener constant}.

\begin{remark}\label{rmk:nunu_int} The definition of $\mathcal{I}_K$ extends naturally to finite signed measures $\nu$ by writing the Jordan--Hahn decomposition of $\nu\otimes\nu$ as
\[
\nu \otimes \nu = \nu^+\otimes\nu^++\nu^-\otimes\nu^--\nu^+\otimes\nu^--\nu^-\otimes\nu^+,
\] where $\nu = \nu^+-\nu^-$ and $\nu^+$, $\nu^-$ are mutually singular non--negative finite measures, assuming that at least one of the sums $(\nu^+\otimes\nu^+)(K)+(\nu^-\otimes\nu^-)(K)$ or $(\nu^+\otimes\nu^-)(K)+(\nu^-\otimes\nu^+)(K)$ is finite. Note that, since $K$ is bounded from below, $|\mathcal{I}_K[\nu]|< \infty$ implies that $K$ is $(\nu\otimes\nu)$--integrable.
\end{remark}

First we will be concerned about the existence and uniqueness of the equilibrium measure. The existence of a measure $\mu^*$ such that $\mathcal{I}_K[\mu^*] = W_K(\M)$ is always guaranteed (see \cite{book}). However, for uniqueness we need a bit more work.

\begin{definition} A kernel $K$ is called \emph{strictly conditionally positive definite} if for every signed finite Borel measure $\nu$ supported on $\M$ such that $\nu(\M) = 0$ and the quantity
\[
\mathcal{I}_K[\nu] := \int_{x,y\in \M}K(x,y)d\nu(x)d\nu(y)
\] is finite, we have $\mathcal{I}_K[\nu] \geq 0$ and $\mathcal{I}_K[\nu] = 0$ if and only if $\nu \equiv 0$.
\end{definition}

Compare this definition to its discrete counterpart \cite[Chapter 31]{CheneyLight}.

\begin{theorem}\cite{book}\label{th:book_1} Suppose that the kernel $K$ is symmetric, lower semicontinuous, strictly conditionally positive definite and that $W_K(\M) < \infty$. If there is some probability measure $\mu^*$ for which the potential $U_K^{\mu^*}(x)$ has a constant finite value $c$, then $\mu^*$ is the unique equilibrium measure and $\mathcal{I}_K[\mu^*] = W_K(\M) = c$.
\end{theorem}

Assuming that there is a unique equilibrium measure $\mu^*$, the following result tells us that the normalized counting measure converges weak$^*$ to $\mu^*$.

\begin{theorem}\cite{book}\label{th:book_2} If $K$ is lower semicontinuous and symmetric, then
\[
\tau_K(\M) = W_K(\M).
\] Moreover, if $\{\omega_N\}_N$ is any sequence of configurations such that
\[
\lim_{N\ra \infty}\frac{E_K(\M)}{N^2} = \tau_K(\M),
\] then every weak$^*$ limit measure $\nu^*$ of the normalized counting measures
\[
\nu(\omega_N) = \frac{1}{N}\sum_{x\in \omega_N}\delta_x
\] is an equilibrium measure for the continuous energy problem.
\end{theorem}

\begin{remark} In particular, the second statement of Theorem \ref{th:book_2} is valid for a sequence $\{\omega^*_N\}_N$ of minimizers for the discrete problem.
\end{remark}

In the next sections we will prove that the kernel $K(x,y) = G(x,y)$, where $G$ is the Green function, is strictly conditionally positive definite (see Proposition \ref{prop:g_posdef}). One might be tempted to use the classical eigenfunction expansion for the Green function
\[
G(x,y) = \sum_{i\geq 1}\frac{\psi_i(x)\psi_i(y)}{\lambda_i},
\] where the $\psi_i$ are some orthonormal basis of eigenfunctions for the Lapace operator and the $\lambda_i$ are the corresponding eigenvalues. Then one would simply write
\[
\int_{x,y\in\M}G(x,y)d\nu(x,y) \stackrel{(*)}{=} \sum_{i\geq 1}\int_{x,y\in\M}\frac{\psi_i(x)\psi_i(y)}{\lambda_i}d\nu(x,y) = \sum_{i\geq 1}\frac{\left(\int_{x\in \M}\psi_i(x)d\nu(x)\right)^2}{\lambda_i}\geq 0.
\] However, the interchange of the sum and the integral in (*) does not directly follow from the Dominated Convergence Theorem and must be carefully justified. We did not find this task to be easy for a general Riemannian manifold and a general signed measure, so we will then use an alternative argument.

\subsection{Mollifiers in Riemannian manifolds} In $\R^n$ we can define the smooth function
\[
\varphi(u) = \left\{
\begin{array}{rl}
C_ne^{-\frac{1}{1-\|u\|^2}} & \text{ if }\|u\| \leq 1,\\
0 & \text{ if }\|u\| > 1,
\end{array}\right.
\] where $C_n$ is some constant making $\int_{\R^n}\varphi = 1$. For every $\eps > 0$, the classical mollifiers are $\varphi_\eps(u) = \eps^{-n}\varphi(u/\eps)$. Each $\varphi_\eps$ is smooth, non--negative, supported on the Euclidean ball $\mathbb{B}(0,\eps)$ and $\int_{\R^n}\varphi_\eps = 1$. Moreover,
\[
\varphi_\eps \stackrel{*}{\rightharpoonup} \delta_0.
\] That is, for every continuous function $f:\R^n \ra \R$ we have that
\[
\lim_{\eps \ra 0}\int_{u\in\R^n}\varphi_\eps(u)f(u)du = f(0).
\] In a Riemannian manifold, we have a sequence of functions playing a similar role. Define
\[
H_\eps(x,y) = \frac{\varphi_\eps(\exp_x^{-1}y)}{\omega_x(y)}, \qquad 0 < \eps < \inj(x).
\] We have to impose $\eps < \inj(x)$ because if $d(x,y) = \inj(x)$, then $\omega_x(y) = 0$ and also $\exp_x^{-1}y$ might not be well--defined. If $\M$ is compact, then $0 < \inj(\M) \leq \inj(x)$ for every point $x\in \M$, and hence we are allowed to just impose $0 < \eps < \inj(\M)$. This will be the case here, since all our manifolds are compact.

\begin{proposition}\label{prop:mollifiers} The functions $H_\eps(x,y)$ have the following properties:
\begin{enumerate}
\item $H_\eps(x,y) \geq 0$.
\item $\text{supp}\,H_\eps = \{(x,y)\in \M \times \M : d(x,y) \leq \eps\}$.
\item $H(x,y) = H(y,x)$.
\item For each $x\in \M$, $\int_{y\in \M}H_\eps(x,y)\dvol(y) = 1$.
\item For every finite signed Borel measure $\nu$ on $\M$ the function
\begin{equation}\label{eq:nu_eps}
\nu_\eps(x) = \int_{y\in \M}H_\eps(x,y)d\nu(y)
\end{equation} is smooth.
\item For every continuous function $f:\M \ra \R$ and for every $x\in \M$,
\[
\lim_{\eps \ra 0}\int_{y\in\M}f(y)H_\eps(x,y)\dvol(y) = f(x).
\]
\end{enumerate}
\end{proposition}

\begin{proof} (1) and (2) are immediate from the definition of $H_\eps$, since from Proposition \ref{prop:density} $\omega_x(y) > 0$ if $d(x,y) < \inj(\M)$.

Note that $\varphi_\eps(u)$ depends only on $\|u\|$ and
\[
\|\exp_x^{-1}y\| = d(x,y) = d(y,x) = \|\exp_y^{-1}x\|,
\] hence $\varphi_\eps(\exp_x^{-1}y) = \varphi_\eps(\exp_y^{-1}x)$ and $\omega_x(y) = \omega_y(x)$ from Proposition \ref{prop:density}, so (3) follows.

For (4) let us compute in normal coordinates around $x$:
\[
\int_{y\in \M}H_\eps(x,y)\dvol(y) = \int_{\mathbb{B}(0,\eps)}\frac{\varphi_\eps(u)}{\sqrt{\det{g_{ij}(u)}}} \sqrt{\det{g_{ij}(u)}} du = \int_{u\in \mathbb{B}(0,\eps)}\varphi_\eps(u)du = 1.
\]

We will prove (5) by induction. Pick some coordinate system $(U,(x^1, ..., x^n))$ and let $x\in U$ be any point. For each index $1 \leq i \leq n$, let $\gamma(t)$ be an integral curve of the vector field $\partial_i$ with $\gamma(0) = x$, and consider the function
\[
g(t) = \int_{y\in \M}H_\eps(\gamma(t),y)d\nu(y).
\] For every $t \in (a,b)$ the function $t\mapsto H_\eps(\gamma(t),y)$ is smooth because $H_\eps$ is smooth. Writing the Jordan--Hahn decomposition of $\nu$ as $\nu = \nu^+-\nu^-$, where $\nu^+$ and $\nu^-$ are (non--negative) finite measures, and applying \cite[Corollary 2.8.7]{bogachev_i}, $g$ is differentiable and
\[
g'(t) = \int_{y\in \M}\frac{d}{dt}H_\eps(\gamma(t),y)d\nu(y).
\] Hence,
\[
\partial_i\nu_\eps(x) = \int_{y\in \M}\partial_iH_\eps(x,y)\,d\nu(y),
\] and thus $\nu_\eps$ is $C^1$. Now assume that $\nu_\eps$ is $C^k$ and that
\[
\partial_{i_1} \cdots \partial_{i_k} \nu_\eps(x) = \int \partial_{i_1} \cdots \partial_{i_k} H_\eps(x,y)d\nu(x)
\] for some indexes $1 \leq i_1, ..., i_k \leq n$. If $1 \leq i \leq n$ is any index, since $x\mapsto \partial_{i_1} \cdots \partial_{i_k} H_\eps(x,y)$ is again a smooth function, then by the base case $\nu_\eps$ is $C^{k+1}$ and
\[
\partial_i \partial_{i_1} \cdots \partial_{i_k} \nu_\eps(x) = \int_{y\in \M}\partial_i\partial_{i_1}\cdots \partial_{i_k}H_\eps(x,y)d\nu(y).
\] Hence $\nu_\eps$ is $C^\infty$.

Finally, we prove (6). Let $f$ be continuous and let $x\in \M$ be any point. If we take normal coordinates in $B(x,\eps)$ for $0 < \eps < \inj(\M)$, we get
\[
f_\eps(x) = \int_{y\in \M}H_\eps(x,y)f(y)\dvol(y) = \int_{u\in \mathbb{B}(0,\eps)}\varphi_\eps(u)f(\exp_xu)du \ra f(\exp_x0) = f(x).
\]
\end{proof}

\subsection{Positive definiteness of the Green function} With the Riemannian mollifiers in hand, now we are able to prove that the Green function is a strictly conditionally positive definite kernel.

\begin{remark} In what follows we will make use of Fubini's Theorem and Lebesgue's Dominated Convergence Theorem for signed finite measures $\nu$. Note that every continuous function $f:\M \ra \R$ is $\nu$-integrable, since
\[
|\nu|(|f|) = \int |f|d\nu^++\int |f|d\nu^-,
\] where $\nu^+$ and $\nu^-$ are both finite (non--negative) measures and $\M$ is compact. The same holds for $\nu \otimes \dvol = \nu^+\otimes \dvol - \nu^-\otimes \dvol$ and continuous $f:\M \times \M \ra \R$.
\end{remark}

\begin{lemma}\label{lem:aprox_nu} Let $\M$ be a compact manifold and let $\nu$ be a signed finite Borel measure such that $\nu(\M) = 0$. The sequence $\{\nu_\eps\}_\eps$ given in \eqref{eq:nu_eps} satisfies:
\begin{enumerate}
\item $\nu_\eps(\M) = 0$, i.e., $\int_{x\in \M}\nu_\eps(x)\dvol(x) = 0$.
\item $\nu_\eps \stackrel{*}{\rightharpoonup} \nu$ as $\eps \ra 0$.
\end{enumerate}
\end{lemma}

\begin{proof} For each $0 < \eps < \inj(\M)$, $\nu_\eps(x)$ is smooth as seen in Proposition \ref{prop:mollifiers}. Then
\begin{align*}
\int_{x\in \M}\nu_\eps(x)\dvol(x) &= \int_{x\in \M}\int_{y\in \M}H_\eps(x,y)d\nu(y)\dvol(x)\\
&= \int_{y\in \M}\left(\int_{x\in \M}H_\eps(x,y)\dvol(x)\right)d\nu(y) \quad \text{(by Fubini's Theorem)}\\
&= \int_{y\in \M} d\nu(y) \quad \text{(by (4) in Proposition \ref{prop:mollifiers})}\\
&= 0.
\end{align*} This proves (1). Now let $f:\M \ra \R$ be a continuous function. Then
\begin{align*}
\int_{y\in \M}f(y)\nu_\eps(y)\dvol(y) &= \int_{y\in \M}f(y)\int_{x\in \M}H_\eps(x,y)d\nu(x)\dvol(y)\\
&= \int_{x\in \M}\left(\int_{y\in \M}f(y)H_\eps(x,y)\dvol(y)\right)d\nu(x) \quad \text{(by Fubini's Theorem)}
\end{align*} By (6) in Proposition \ref{prop:mollifiers},
\[
\lim_{\eps\ra 0}\int_{y\in \M}f(y)H_\eps(x,y)\dvol(y) = f(x).
\] Also, since $f$ is continuous and $\M$ is compact, there is some constant $C$ such that
\[
\left|\int_{y\in \M}f(y)H_\eps(x,y)\dvol(y)\right| \leq C\int_{y\in \M}|H_\eps(x,y)|\dvol(y) = C,
\] (from (1) and (4) in Proposition \ref{prop:mollifiers}). The constant $C$ is a $\nu$--integrable function and, by the Dominated Convergence Theorem,
\[
\lim_{\eps\ra 0}\int_{y\in \M}f(y)\nu_\eps(y)\dvol(y) = \int_{x\in \M}f(x)d\nu(x),
\] proving (2).
\end{proof}

\begin{lemma}\label{lem:lemita1} There is a constant $C > 0$ such that, for every $0 < \eps < \frac{\inj(\M)}{4}$, and for every pair of points $y,z\in \M$ such that $d(y,z) < 2\eps$, we have
\[
\int_{x\in \M}H_\eps(x,z)\log d(x,y)^{-1}\dvol(x) \leq C\log d(z,y)^{-1} \quad \text{ if }n = 2,
\]
\[
\int_{x\in \M}\frac{H_\eps(x,z)}{d(x,y)^{n-2}}\dvol(x) \leq \frac{C}{d(z,y)^{n-2}} \quad \text{ if }n > 2.
\]
\end{lemma}   

\begin{proof} We prove the case $n > 2$ (the case $n = 2$ is similar). Let $0 < \eps < \frac{\inj(\M)}{4}$, and let $y,z\in \M$ be a a pair of points with $d(y,z) < 2\eps$. The function
\[
(x,z) \mapsto \omega_x(z)
\] is continuous and strictly positive if $d(x,z) < \inj(\M)$, hence it is continuous and strictly positive on the compact set
\[
\left\{(x,z)\in \M \times \M : d(x,z) \leq \frac{3}{4}\inj(\M)\right\}.
\] Therefore, there exist constants $0 < k \leq K$, not depending on $\eps$, such that for every $p,q\in \M$ with $d(p,q) \leq 3\eps$, we have $0 < k \leq \omega_p(q) \leq K$. Since $\varphi_\eps(u)$ reaches a maximum at $u = 0$,
\begin{equation}\label{eq:h1}
H_\eps(x,z) = \frac{\varphi_\eps(\exp_z^{-1}x)}{\omega_z(x)} \leq \frac{1}{k}\varphi_\eps(0) = \frac{C_n}{ek\eps^n}.
\end{equation} (Recall that $C_n$ is the constant making $\int_{\R^n}\varphi = 1$). Then,
\[
\int_{x\in \M}\frac{H_\eps(x,z)}{d(x,y)^{n-2}}\dvol(x) \leq \frac{C_n}{ek\eps^n}\int_{x\in B(z,\eps)}\frac{1}{d(x,y)^{n-2}}\dvol(x).
\] Since $\eps < \frac{\inj(\M)}{3}$, there are normal coordinates around $y$ defined on $B(y,3\eps)$. Since $d(y,z) < 2\eps$, if $x\in B(z,\eps)$ then by the triangle inequality $d(x,y) < 3\eps$, so $B(y,3\eps)\supset B(z,\eps)$. Now, taking normal coordinates on $B(y,3\eps)$,
\begin{align*}
\int_{x\in B(z,\eps)}\frac{1}{d(x,y)^{n-2}}\dvol(x) &\leq \int_{x\in B(y,3\eps)}\frac{1}{d(x,y)^{n-2}}\dvol(x)\\
&= \int_{u\in \mathbb{B}(3\eps)}\frac{\sqrt{\det{g_{ij}(u)}}}{\|u\|^{n-2}}du\\
&= \int_{u\in \mathbb{B}(3\eps)}\frac{\omega_y(\exp_xu)}{\|u\|^{n-2}}du\\
&\leq K\int_{u\in \mathbb{B}(3\eps)}\frac{1}{\|u\|^{n-2}}du\\
&= K\int_0^{3\eps}\int_{\theta\in \mathbb{S}^{n-1}(r)}\frac{1}{r^{n-2}}d\theta dr\\
&= K\vol(\S^{n-1})\int_0^{3\eps} rdr\\
&= \frac{9K\vol(\S^{n-1})\eps^2}{2}.
\end{align*} %Together with \eqref{eq:h1} we get that
We then conclude:
\[
\int_{x\in \M}\frac{H_\eps(x,z)}{d(x,y)^{n-2}}\dvol(x) \leq \frac{9C_nK\vol(\S^{n-1})}{2ek\eps^{n-2}} \stackrel{d(y,z) < 2\eps}{\leq} \frac{2^{n-3}\cdot 9 C_nK\vol(\S^{n-1})}{ekd(y,z)^{n-2}}.
\] The result follows by taking
\[
C = \frac{2^{n-3}\cdot 9 C_nK\vol(\S^{n-1})}{ek}.
\]
\end{proof}

\begin{lemma}\label{lem:lemita2} Let $0 < \eps < \inj(\M)$ and let $y,z,w \in \M$ be points such that $d(y,w) \leq \eps$ and $d(y,z) > 2\eps$. Then
\[
\int_{x\in \M}H_\eps(x,z)\log(x,y)^{-1}\dvol(x) \leq \log d(y,w)^{-1} \quad \text{ if }n = 2,
\]
\[
\int_{x\in \M}\frac{H_\eps(x,z)}{d(x,y)^{n-2}}\dvol(x) \leq \frac{1}{d(y,w)^{n-2}} \quad \text{ if }n>2.
\]
\end{lemma}

\begin{proof} We prove the case $n > 2$ (the case $n = 2$ is similar). From the triangle inequality, for any $x\in B(z,\eps)$ we have that $d(x,y) \geq d(y,w)$. Hence
\[
\int_{x\in \M}\frac{H_\eps(x,z)}{d(x,y)^{n-2}}\dvol(x) \leq \int_{x\in \M}\frac{H_\eps(x,z)}{d(y,w)^{n-2}}\dvol(x) = \frac{1}{d(y,w)^{n-2}},
\] since $\int H_\eps = 1$.
\end{proof}

\begin{lemma}\label{lem:lemita3} There exists a positive constant $K$ such that, for every $0 < \eps < \frac{\inj(\M)}{4}$ and for every $z,w\in \M$ with $z\neq w$,
\[
\int_{x,y\in \M}H_\eps(x,z)H_\eps(y,w)\log d(x,y)^{-1}\dvol(x)\dvol(y) \leq K\log d(z,w)^{-1} \quad \text{ if }n = 2,
\]
\[
\int_{x,y\in \M}\frac{H_\eps(x,z)H_\eps(y,w)}{d(x,y)^{n-2}}\dvol(x)\dvol(y) \leq \frac{K}{d(z,w)^{n-2}} \text{ if }n > 2.
\]
\end{lemma}

\begin{proof} We prove the case $n > 2$ (the case $n = 2$ is similar). Let $0 < \eps < \frac{\inj(\M)}{4}$, and let $z,w\in \M$ with $z \neq w$. From Fubini's Theorem and lemmas \ref{lem:lemita1} and \ref{lem:lemita2}, we have that
\begin{align*}
\int_{x,y\in \M}&\frac{H_\eps(x,z)H_\eps(y,w)}{d(x,y)^{n-2}}\dvol(x,y) = \int_{y\in B(w,\eps)}H_\eps(y,w)\left(\int_{x\in \M}\frac{H_\eps(x,z)}{d(x,y)^{n-2}}\dvol(x)\right)\dvol(y)\\
&= \int_{\{y:d(y,z)\leq 2\eps\}\cap B(w,\eps)}H_\eps(y,w)\left(\int_{x\in \M}\frac{H_\eps(x,z)}{d(x,y)^{n-2}}\dvol(x)\right)\dvol(y)\\
&\quad +\int_{\{y:d(y,z)> 2\eps\}\cap B(w,\eps)}H_\eps(y,w)\left(\int_{x\in \M}\frac{H_\eps(x,z)}{d(x,y)^{n-2}}\dvol(x)\right)\dvol(y)\\
&\leq C\int_{\{y:d(y,z)\leq 2\eps\}}\frac{H_\eps(y,w)}{d(y,z)^{n-2}}\dvol(y)\\
&+\quad\int_{\{y:d(y,z)> 2\eps\}}\frac{H_\eps(y,w)}{d(y,w)^{n-2}}\dvol(y)\\
\end{align*} If $d(z,w) \leq 2\eps$, applying Lemma \ref{lem:lemita1} we get a bound for the first of the integrals above:
\[
C\int_{\{y:d(y,z)\leq 2\eps\}}\frac{H_\eps(y,w)}{d(y,z)^{n-2}}\dvol(y) \leq C\int_{y\in \M}\frac{H_\eps(y,w)}{d(y,z)^{n-2}}\dvol(y) \leq \frac{C^2}{d(z,w)^{n-2}},
\] and for the second one we have that
\begin{align*}
\int_{\{y:d(y,z)> 2\eps\}}\frac{H_\eps(y,w)}{d(y,w)^{n-2}}\dvol(y) &\leq \int_{y\in \M}\frac{H_\eps(y,w)}{d(y,w)^{n-2}}\dvol(y)\\
&= \int_{u\in \mathbb{B}(\eps)}\frac{\varphi_\eps(u)}{\|u\|^{n-2}}du\\
&\leq \frac{C_n}{e\eps^n}\int_{u\in \mathbb{B}(\eps)}\frac{1}{\|u\|^{n-2}}du\\
&= \frac{C_n\vol(\S^{n-1})}{2e\eps^{n-2}}\\
&\leq \frac{2^{n-3}C_n\vol(\S^{n-1})}{ed(z,w)^{n-2}}.
\end{align*} Therefore, if $d(z,w) \leq 2\eps$, then
\[
\int_{x,y\in \M}\frac{H_\eps(x,z)H_\eps(y,w)}{d(x,y)^{n-2}}\dvol(x)\dvol(y) \leq \frac{C^2+2^{n-3}C_n\vol(\S^{n-1})e^{-1}}{d(z,w)^{n-2}}.
\] Now assume that $d(z,w) > 2\eps$. From the triangle inequality, if $H_\eps(x,z)H_\eps(y,w)\neq0$,
\[
d(x,y) \geq d(z,w)-d(z,x)-d(y,w) > 0,
\] hence
\[
\int_{x,y\in \M}\frac{H_\eps(x,z)H_\eps(y,w)}{d(x,y)^{n-2}}\dvol(x,y) \leq \int_{x,y\in \M}\frac{H_\eps(x,z)H_\eps(y,w)}{(d(z,w)-d(z,x)-d(y,w))^{n-2}}\dvol(x,y).
\] Let us denote $t = d(z,w)$. Taking normal coordinates around $z$ and $w$, this last integral equals
\[
\int_{u,v\in \mathbb{B}(\eps)}\frac{\varphi_\eps(u)\varphi_\eps(v)}{(t-\|u\|-\|v\|)^{n-2}}dudv = \frac{C_n^2}{\eps^{2n}}\int_{u,v\in\mathbb{B}(\eps)}\frac{e^{-\frac{1}{1-\frac{\|u\|^2}{\eps^2}}}e^{-\frac{1}{1-\frac{\|v\|^2}{\eps^2}}}}{(t-\|u\|-\|v\|)^{n-2}}dudv
\] Taking polar coordinates, this equals
\[
\frac{C_n^2\vol(\S^{n-1})^2}{\eps^{2n}}\int_0^\eps\int_0^\eps \frac{r^{n-1}s^{n-1}e^{-\frac{1}{1-\frac{r^2}{\eps^2}}}e^{-\frac{1}{1-\frac{s^2}{\eps^2}}}}{(t-r-s)^{n-2}}drds
\] Changing variables $x = r/\eps$, $y = s/\eps$, that is
\begin{align*}
\frac{C_n^2\vol(\S^{n-1})^2}{\eps^{2n}}\int_0^1\int_0^1 \frac{(\eps x)^{n-1}(\eps y)^{n-1}e^{-\frac{1}{1-x^2}}e^{-\frac{1}{1-y^2}}}{(t-\eps x-\eps y)^{n-2}}&\eps^2dxdy\\
= C_n^2\vol(\S^{n-1})^2\int_0^1\int_0^1 \frac{x^{n-1}y^{n-1}e^{-\frac{1}{1-x^2}}e^{-\frac{1}{1-y^2}}}{(t-\eps x -\eps y)^{n-2}}dxdy\\
 \leq C_n^2\vol(\S^{n-1})^2 \int_0^1 \int_0^1 \frac{x^{n-1}y^{n-1}e^{-\frac{1}{1-x^2}}e^{-\frac{1}{1-y^2}}}{\left(t-\frac{t}{2}x-\frac{t}{2}y\right)^{n-2}}dxdy \qquad (t > 2\eps)\\
\end{align*} and this last integral equals $\frac{\tilde{K}}{t^{n-2}}$ with
\[
\tilde{K} = C_n^2\vol(\S^{n-1})^2 \int_0^1 \int_0^1 \frac{x^{n-1}y^{n-1}e^{-\frac{1}{1-x^2}}e^{-\frac{1}{1-y^2}}}{(1-x/2-y/2)^{n-2}}dxdy.
\] Let us see that the integral in the right hand side converges. We have that
\begin{align*}
\int_0^1 \int_0^1 \frac{x^{n-1}y^{n-1}e^{-\frac{1}{1-x^2}}e^{-\frac{1}{1-y^2}}}{(1-x/2-y/2)^{n-2}}dxdy &\leq \int_0^1\int_0^1 \frac{e^{-\frac{1}{1-x^2}}e^{-\frac{1}{1-y^2}}}{(1-x/2-1/2)^{n-2}}dxdy\\
&= 2^{n-2}\int_0^1 e^{-\frac{1}{1-y^2}}dy\int_0^1 \frac{e^{-\frac{1}{1-x^2}}}{(1-x)^{n-2}}dx\\
&= \frac{2^{n-3}}{C_1}\int_0^1 \frac{e^{-\frac{1}{1-x^2}}}{(1-x)^{n-2}}dx
\end{align*} This integral converges because
\[
\lim_{x\ra 1^-}\frac{e^{-\frac{1}{1-x^2}}}{(1-x)^{n-2}} = 0.
\]

Finally, taking
\[
K = \max\{\tilde{K},C^2+2^{n-3}C_ne^{-1}\vol(\S^{n-1})\}
\] the result follows.
\end{proof}

\begin{lemma}\label{lem:lemita4} There exists constants $C_1, C_2 > 0$ such that, for every $x \neq y$,
\[
\log d(x,y)^{-1} \leq C_1G(x,y)+C_2 \qquad \text{ if }n=2,
\]
\[
\frac{1}{d(x,y)^{n-2}} \leq C_1G(x,y)+C_2 \qquad \text{ if }n > 2.
\]
\end{lemma}

\begin{proof} The result follows from \cite[Proposition 6.1]{ponge}.
\end{proof}

\begin{lemma}\label{lem:Gstar} Let $\nu$ be a signed finite Borel measure such that $\nu(\M) = 0$ and $|\mathcal{I}_G[\nu]| < \infty$, and let $\{\nu_\eps\}_\eps$ be the sequence of smooth functions in Lemma \ref{lem:aprox_nu}. Then
\[
\lim_{\eps \ra 0}\int_{x,y\in \M}G(x,y)\nu_\eps(x)\nu_\eps(y)\dvol(x)\dvol(y) = \int_{x,y\in\M}G(x,y)d\nu(x)d\nu(y).
\]
\end{lemma}

\begin{proof} Applying Fubini's Theorem, we have that
\begin{align*}
\int_{x,y\in \M}&G(x,y)\nu_\eps(x)\nu_\eps(y)\dvol(x)\dvol(y)\\
&= \int_{z,w\in \M}\left(\int_{x,y\in \M}G(x,y)H_\eps(z,x)H_\eps(w,y)\dvol(x)\dvol(y)\right)d\nu(z)d\nu(w).
\end{align*} Let us denote
\[
f_\eps(z,w) = \int_{x,y\in \M}G(x,y)H_\eps(z,x)H_\eps(w,y)\dvol(x)\dvol(y),
\] and let us see that we may apply the Dominated Convergence Theorem.

First, if $z\neq w$, choosing $\eps > 0$ small enough the closed balls $\overline{B}(z,\eps)$ and $\overline{B}(w,\eps)$ do not intersect. The integrand in $f_\eps(z,w)$ has support on $F = \overline{B}(z,\eps)\times \overline{B}(w,\eps)$, and restricted to this closed set the function $G$ is continuous. By Tietze's Extension Theorem, we may choose a function $\tilde{G}$, continuous con $\M \times \M$, such that $\tilde{G}|_F = G$. By (6) in Proposition \ref{prop:mollifiers} and applying Fubini's theorem,
\[
\lim_{\eps\ra 0}f_\eps(z,w) = \lim_{\eps \ra 0}\int_{x,y\in \M}\tilde{G}(x,y)H_\eps(z,x)H_\eps(w,y)\dvol(x)\dvol(y) = \tilde{G}(z,w) = G(z,w).
\] This proves that $f_\eps(z,w) \ra G(z,w)$ pointwise almost everywhere on $\M \times \M$.

Suppose that $n > 2$ (the proof for the case $n = 2$ is similar). From  (2) in Theorem \ref{thm:green_f}, there exists a positive constant $k$ such that, for every $x,y\in \M$,
\[
|G(x,y)| \leq \frac{k}{d(x,y)^{n-2}}.
\] Hence,
\begin{align*}
|f_\eps(z,w)| &\leq \int_{x,y\in \M}\frac{kH_\eps(z,x)H_\eps(w,y)}{d(x,y)^{n-2}} \dvol(x)\dvol(y)\\
&\leq \frac{C}{d(z,w)^{n-2}} \quad \text{(by Lemma \ref{lem:lemita3})}\\
&\leq CC_1G(z,w)+CC_2, \quad \text{(by Lemma \ref{lem:lemita4})},
\end{align*} with all the constants independent from $\eps$. Since $G$ is $\nu\otimes\nu$--integrable (see Remark \ref{rmk:nunu_int}) and so are the constant functions, the result follows from the Dominated Convergence Theorem.
\end{proof}

\begin{proposition}\label{prop:g_posdef} The kernel $G$ is strictly conditionally positive definite.
\end{proposition}

\begin{proof} Let $\nu$ be a signed finite Borel measure on $\M$ such that $\nu(\M) = 0$ and $\mathcal{I}_K[\nu]$ is finite. Let $\{\nu_\eps\}_\eps$ with $0 < \eps < \frac{\inj(\M)}{4}$ be the sequence of smooth functions from Lemma \ref{lem:aprox_nu}. For each $\eps$, let $\tilde{\nu}_\eps$ be a zero mean function such that $\Delta\tilde{\nu}_\eps = \nu_\eps$ (which exists from Remark \ref{rmk:poisson}). Then, applying Fubini's Theorem and (1) from Theorem \ref{thm:green_f},
\begin{align*}
\int_{x,y\in \M}G(x,y)\nu_\eps(x)\nu_\eps(y)\dvol(x)\dvol(y) &= \int_{y\in \M}\left(\int_{x\in \M}G(x,y)\nu_\eps(x)\dvol(x)\right)\nu_\eps(y)\dvol(y)\\
&= \int_{y\in \M}\tilde{\nu}_{\eps}(y)\Delta\tilde{\nu}_\eps(y)\dvol(y)\\
&= \int_{y\in \M}\|\nabla\tilde{\nu}_\eps(y)\|^2\dvol(y) \geq 0,
\end{align*} where we have used Green's First Identity. From Lemma \ref{lem:Gstar},
\[
\int_{x,y\in\M}G(x,y)d\nu(x)d\nu(y) = \lim_{\eps \ra 0}\int_{x,y\in \M}G(x,y)\nu_\eps(x)\nu_\eps(y)\dvol(x)\dvol(y) \geq 0.
\]

Now assume that
\[
\int_{x,y\in \M}G(x,y)d\nu(x)d\nu(y) = 0.
\] Let $u\in C^2(\M)$ be any zero mean function. Then,
\begin{align*}
\left|\int_{y\in \M}u(y)d\nu(y)\right| &= \left|\lim_{\eps\ra 0}\int_{y\in\M}u(y)\nu_\eps(y)\dvol(y)\right|\\
&= \left|\lim_{\eps\ra 0}\int_{y\in \M}\left(\int_{x\in \M}G(x,y)\Delta u(x)\dvol(x)\right)\nu_{\eps}(y)\dvol(y)\right|\\
&= \left|\lim_{\eps\ra 0}\int_{x\in \M}\Delta u(x)\left(\int_{y\in \M}G(x,y)\nu_\eps(y)\dvol(y)\right)\dvol(x)\right|\\
&= \left|\lim_{\eps\ra 0}\int_{x\in \M}\Delta u(x)\tilde{\nu}_\eps(x)\dvol(x)\right|\\
&= \left|\lim_{\eps\ra 0}\int_{x\in \M}g(\nabla u(x),\nabla\tilde{\nu}_\eps(x))\dvol(x)\right|\\
&\leq \left(\int_{x\in \M}\|\nabla u(x)\|^2\dvol(x)\right)^{1/2}\lim_{\eps \ra 0}\left(\int_{x\in\M}\|\nabla\tilde{\nu}_\eps(x)\|^2\dvol(x)\right)^{1/2},
\end{align*} and
\[
\lim_{\eps \ra 0}\int_{x\in \M}\|\nabla\tilde{\nu}_\eps(x)\|^2\dvol(x) = \int_{x,y\in \M}G(x,y)d\nu(x)d\nu(y) = 0.
\] Hence, for every zero mean function $u\in C^2(\M)$,
\[
\int_{y\in \M}u(y)d\nu(y) = 0,
\]
and by adding a constant it is immediate to see that the zero mean hypotheses is unnecessary. Now, let $f:\M \ra \R$ be a continuous function. Since $\M$ is compact, by the Stone-Weierstrass Theorem there is a sequence $\{f_n\}_n$ of $C^2$ functions converging uniformly to $f$ on $\M$. Thus, for every $x\in \M$ and for all $n$ sufficiently large, $|f_n(x)-f(x)|\leq 1$, so
\[
|f_n(x)| = |f_n(x)-f(x)+f(x)| \leq |f_n(x)-f(x)|+|f(x)| \leq 1+C,
\] where $C$ is a global bound for $f$. Therefore, by the Dominated Convergence Theorem,
\[
\int_{y\in \M}f(y)d\nu(y) = \lim_n\int_{y\in \M}f_n(y)d\nu(y) = 0.
\] Hence, $\nu \equiv 0$ and the result follows.
\end{proof}

\subsection{Proof of Theorem \ref{th:main}} We have shown in Proposition \ref{prop:g_posdef} that the kernel $G$ is strictly conditionally positive definite. Moreover, $G$ is symmetric, lower semicontinuous and
\[
W_G(\M) = \inf_{\mu}\mathcal{I}_G[\mu] \leq V^{-2}\int_{x,y\in\M}G(x,y)\dvol(x)\dvol(y) = 0 < \infty.
\] Since for the measure $\lambda = V^{-1}\dvol$ the potential
\[
U_G^\lambda(x) = V^{-1}\int_{y\in\M}G(x,y)\dvol(y)
\] has a constant finite value (namely, $0$), by Theorem $\ref{th:book_1}$ the normalized Riemannian measure $\lambda$ is the unique equilibrium measure for $G$ and $W_G(\M) = 0$. By Theorem \ref{th:book_2}, any convergent subsequence of $\frac{1}{N}\sum_{x\in \omega^*_N}\delta_x$ converges to $\lambda$. Finally, the result follows from the Banach--Alaoglu Theorem.
\begin{flushright}\qedsymbol\end{flushright}

\bibliographystyle{amsplain}
\bibliography{references}

\end{document}